\documentclass[reqno,12pt]{amsart}

\pdfoutput=1

\usepackage{color} 
\usepackage{ifpdf}
\ifpdf
    \usepackage[pdftex]{graphicx}
    \usepackage[pdftex]{hyperref}
    \hypersetup{
        unicode=false,          
        pdftoolbar=true,        
        pdfmenubar=true,        
        pdffitwindow=false,     
        pdfstartview={FitH},    
        pdftitle={MCP Article},      
        pdfauthor={Michael Holst},   
        pdfsubject={Mathematics},    
        pdfcreator={Michael Holst},  
        pdfproducer={Michael Holst}, 
        pdfkeywords={PDE, analysis, mathematical physics}, 
        pdfnewwindow=true,      
        colorlinks=true,        
        linkcolor=red,          
        citecolor=blue,         
        filecolor=magenta,      
        urlcolor=cyan           
    }

\else
    \usepackage{graphicx}
    \usepackage{pstricks}

\fi

\usepackage{times}
\usepackage{amsfonts}

\usepackage{amsmath}
\usepackage{amsthm}
\usepackage{amssymb}
\usepackage{amsbsy}
\usepackage{amscd}

\usepackage{enumerate}
\usepackage{verbatim}
\usepackage{subfigure}

\newtheorem{theorem}{Theorem}[section]

\newtheorem{lemma}[theorem]{Lemma}

\newtheorem{definition}[theorem]{Definition}

\newtheorem{remark}[theorem]{Remark}

\numberwithin{equation}{section}  

  \newcounter{mnote}
  \let\oldmarginpar\marginpar
    \renewcommand\marginpar[1]{\-\oldmarginpar[\raggedleft\footnotesize #1]%
    {\raggedright\footnotesize #1}}

\definecolor{myblue}{rgb}{0.2,0.2,0.7}
\definecolor{mygreen}{rgb}{0,0.6,0}
\definecolor{mycyan}{rgb}{0,0.6,0.6}
\definecolor{myred}{rgb}{0.9,0.2,0.2}
\definecolor{mymagenta}{rgb}{0.9,0.2,0.9}
\definecolor{mywhite}{rgb}{1.0,1.0,1.0}
\definecolor{myblack}{rgb}{0.0,0.0,0.0}

\newcommand{\mcal}{\mathcal}

\newcommand{\beq}{\begin{equation}}
\newcommand{\eeq}{\end{equation}}
\newcommand{\beqa}{\begin{eqnarray}}
\newcommand{\eeqa}{\end{eqnarray}}
\renewcommand{\div}{{\operatorname{div}}}

\renewcommand{\mcal}{\mathcal}
\newcommand{\bs}{\boldsymbol}
\renewcommand{\div}{\operatorname{div}}
\newcommand{\curl}{\operatorname{curl}}
\newcommand{\grad}{\operatorname{grad}}

\begin{document}

\title[Convergence and Optimality of Adaptive Mixed FE Methods]
      {Convergence and Optimality of \\ Adaptive Mixed Finite Element Methods}

\author{Long Chen}
\email{chenlong@math.uci.edu}
\address{Department of Mathematics, University of California at Irvine, Irvine, CA 92697}

\author{Michael Holst}
\email{mholst@math.ucsd.edu}
\address{Department of Mathematics, University of California at San Diego, La Jolla, CA 92093} 

\author{Jinchao Xu}
\email{xu@math.psu.edu}
\address{The School of Mathematical Science, Peking University, and Department of Mathematics, Pennsylvania State University, University Park, PA 16801} 

\thanks{The first two authors were supported in part by NSF Awards 0411723 and 022560, in part by DOE Awards DE-FG02-04ER25620 and DE-FG02-05ER25707, and in part by NIH Award P41RR08605.}
\thanks{The third author was supported in part by NSF DMS-0619587, DMS-0609727, NSFC-10528102 and Alexander Humboldt foundation.}

\date{April 3, 2006}


\begin{abstract}
  The convergence and optimality of adaptive mixed finite element
  methods for the Poisson equation are established in this paper. The
  main difficulty for mixed finite element methods is the lack of
  minimization principle and thus the failure of orthogonality. A
  quasi-orthogonality property is proved using the fact that the error
  is orthogonal to the divergence free subspace, while the part of the
  error that is not divergence free can be bounded by the data oscillation
  using a discrete stability result. This discrete stability result
  is also used to get a localized discrete upper bound which is
  crucial for the proof of the optimality of the adaptive
  approximation.
\end{abstract}

\maketitle


{\footnotesize
\tableofcontents
}

\section{introduction}

Adaptive methods are now widely used in scientific computation to
achieve better accuracy with minimum degrees of freedom. While these
methods have been shown to be very successful, the theory ensuring the
convergence of the algorithm and the advantages over non-adaptive
methods is still under development. Recently, several results have
been obtained for standard finite element methods for elliptic partial
differential equations
\cite{Babuska.I;Vogelius.M1984,Dorfler.W1996,Morin.P;Nochetto.R;Siebert.K2000,Morin.P;Nochetto.R;Siebert.K2003,Binev.P;Dahmen.W;DeVore.R2004,Stevenson.R2007,Morin.P;Siebert.K;Veeser.A2007a,Cascon.J;Kreuzer.C;Nochetto.R;Siebert.K2007,Chen.L;Xu.J2007}.

In this paper, we shall establish the convergence and optimality of
adaptive mixed finite element methods (AMFEMs) of the model problem
\begin{equation}\label{eq:Poisson}
-\Delta u=f \quad \hbox{in } \Omega, \quad \hbox{ and }\quad
 u=0 \quad \hbox{on } \partial \Omega,
\end{equation}
posed on a polygonal and simply connected domain $\Omega \subset
\mathbb R^2$. In many applications (\cite{Brezzi.F;Fortin.M1991}) the
variable $\boldsymbol \sigma =-\nabla u$ is of interest 
and it is therefore convenient to use mixed finite element methods,
such as the Raviart-Thomas mixed method ~\cite{Raviart.P;Thomas.J1977}
and Brezzi-Douglas-Marini mixed
method~\cite{Brezzi.F;Douglas.J;Marini.L1985}. We shall construct
adaptive mixed finite element methods based on the local refinement of
triangulations and prove they will produce a sequence of approximation
of $\bs \sigma$ in an optimal way.

Our main result is the following optimal convergence of our algorithms
{\bf \small AMFEM} and its variant. Let $\boldsymbol \sigma _N$ be the
approximation of $\boldsymbol \sigma$ based on the triangulation
$\mathcal T_N$ obtained in {\bf \small AMFEM}. If $\boldsymbol \sigma
\in \mathcal A^s$ and $f\in \mathcal A^s_o$, then
\begin{equation}\label{finalopt}
\|\boldsymbol \sigma -\boldsymbol \sigma _N\|\leq C(\|\boldsymbol \sigma\|_{\mathcal A^s}+\|f\|_{\mathcal A^s_o})(\# \mathcal T_N-\#\mathcal T_0)^{-s},
\end{equation}
where $(\mathcal A^s, \|\cdot\|_{\mathcal A^s})$ and $(\mathcal
A^s_o,\|\cdot\|_{\mathcal A^s_o})$ are approximation spaces as in
\cite{Binev.P;Dahmen.W;DeVore.R2004}. The index $s$ is used to
characterize the best possible approximation rate of $\boldsymbol
\sigma$, which depends on the regularity of the solution and data, and
the order of elements. For example when $f\in L^2(\Omega)$ and
$\boldsymbol \sigma \in \boldsymbol W^{1,1}(\Omega)$, we can achieve the
optimal convergence rate $s=1/2$ for the lowest order Raviart-Thomas
finite element space. We refer
to~\cite{Binev.P;Dahmen.W;DeVore.R;Petrushev.P2002} for the
characterization of $\mathcal A^s$ in terms of Besov spaces and to
\cite{Bacuta.C;Bramble.J;Xu.J2002,Bacuta.C;Bramble.J;Xu.J2003,Dahlke.S;DeVore.R1997,Dahlke.S1999}
for the regularity results in Besov norms. We comment that to apply our adaptive
algorithm, we do not need to know $s$ explicitly. Our algorithm will
produce the best possible approximation rate for the unknown $\bs
\sigma$.

For the analysis of the convergence of adaptive procedure, we follow
the new approach by Cascon, Kreuzer, Nochetto and Siebert
\cite{Cascon.J;Kreuzer.C;Nochetto.R;Siebert.K2007}, and for the
optimality we mainly use the simplified case in Stevenson's
work~\cite{Stevenson.R2007}. A distinguish feature of the new
approach for the convergence proof is the relaxation of the interior
node requirement for the refinement. We do not claim any originality
on the proof of convergence and optimality. Instead the main
contribution of this paper is to establish two important ingredients
used in the proof, namely quasi-orthogonality and discrete upper
bound.

One main ingredient in the convergence analysis of standard AFEM is
that the error is orthogonal to the finite element spaces in
energy-related inner product since the standard finite element
approximation can be characterized as a minimizer of Dirichlet-type
energy. For mixed finite element methods, however, the approximation
is a saddle point of the corresponding energy and thus there is no
orthogonality available. We shall prove a quasi-orthogonality
result. A similar result for the lowest order Raviart-Thomas finite
element space has recently been proved by Carstensen and Hoppe
\cite{Carstensen.C;Hoppe.R2006}, where a special relation between
mixed finite element method and non-conforming method is used. In this
paper, we shall propose a new and more straight-forward approach which
works for any order elements and both Raviart-Thomas and
Brezzi-Douglas-Marini methods. The main observation is that the error
is orthogonal to the divergence free subspace, while the part of the
error containing divergence can be bounded by the data oscillation
using a discrete stability result.

Another ingredient to establish the optimality of the adaptive
algorithm is the localized discrete upper bound for \textit{a
  posteriori} error estimator. Using the discrete stability result, we
are able to obtain such discrete upper bound and use it to prove the
optimality of the convergent algorithm. The optimality of mixed
adaptive finite element methods seems to be new.

The rest of this paper is organized as follows. In Section 2, we shall
introduce mixed finite element methods and give a short review of mesh
adaptivity through local refinement. We shall include many preliminary
results in this section for later usage. In Section 3, we shall prove
the discrete stability result and use it to prove the
quasi-orthogonality result. In Section 4, we shall present \textit{a
  posteriori} error estimator and prove the discrete upper bound. In
Section 5, we shall present our algorithms and prove their convergence
and optimality.

Throughout this paper, we shall use standard notation for Sobolev
spaces and use boldface letter for the spaces of vectors. The letter
$C$, without subscript, denotes generic constants that may not be the
same at different occurrences and $C_i$, with subscript, denotes
specific important constants.

\section{Preliminaries}
In this section we shall introduce mixed finite element methods for
the Poisson equation and discuss the general
procedure of adaptive methods through local refinement. We shall also
include a result on the approximation of the data.

\subsection{Mixed finite element methods}
The standard finite element method involves writing (\ref{eq:Poisson})
as a primal variational formulation: for a given $f\in L^2(\Omega)$,
find $u\in H_0^1(\Omega)$ such that
\begin{equation}\label{primary}
  \int _{\Omega}\nabla u \cdot \nabla v= \int _{\Omega}f v \qquad \forall v\in H_0^1(\Omega),
\end{equation}
and then finding an approximation by solving (\ref{primary}) in
finite-dimensional subspaces of $H_0^1(\Omega)$. In many applications
(\cite{Brezzi.F;Fortin.M1991}) the variable $\boldsymbol \sigma
=-\nabla u$ is of interest, and it is therefore convenient to use
mixed finite element methods. Let us first write (\ref{eq:Poisson}) as
a first order system:
\begin{equation}\label{eq:MixPoisson}
\boldsymbol \sigma + \nabla u=0,\;
{\rm div\,} \boldsymbol \sigma =f\; \hbox{ in } \Omega,
\quad \hbox{ and }\quad  
u=0\; \hbox{ on } \partial \Omega.
\end{equation}
Let 
$$
\boldsymbol \Sigma =\boldsymbol H({\rm div\,};\Omega):=\{\boldsymbol
\tau \in \boldsymbol L^2(\Omega): {\rm div\,} \boldsymbol \tau \in
L^2(\Omega)\}, \; \hbox{ and }\; U=L^2(\Omega).
$$
We shall use $\|\cdot\|$ to denote $L^2$-norm and $\|\cdot
\|_{H(\div)}$ for the $\bs H(\div)$ norm:
$$
\|\bs \tau\|_{H(\div)} = (\|\bs \tau\|^2+\|\div \bs \tau\|^2)^{1/2},
\quad \forall \tau \in \bs \Sigma.
$$

The mixed (or dual) variational formulation of (\ref{eq:MixPoisson})
is, given an $f\in L^2(\Omega)$, find $(\boldsymbol \sigma,u) \in
\boldsymbol \Sigma \times U$ such that
\begin{align}
  (\boldsymbol \sigma, \boldsymbol \tau)-({\rm div\,}\boldsymbol \tau, u)&=0  & \forall \boldsymbol \tau \in \boldsymbol \Sigma, \label{eq:mixweak1}\\
  ({\rm div\,}\boldsymbol \sigma, v)&=(f,v) &\forall v\in
  U, \label{eq:mixweak2}
\end{align}
where $(\cdot,\cdot)$ is the inner product for $L^2(\Omega)$ or
$\boldsymbol L^2(\Omega)$. Note that the Dirichlet boundary condition
is imposed as a natural boundary condition in the dual formulation
(\ref{eq:mixweak1}) using integration by parts. The existence and
uniqueness of the solution $(\boldsymbol \sigma, u)$ to
(\ref{eq:mixweak1})-(\ref{eq:mixweak2}) follows from the so-called
inf-sup condition which can be easily established for this model
problem \cite{Brezzi.F;Fortin.M1991}.

Given a shape regular and conforming (in the sense of
\cite{Ciarlet.P1978}) triangulation $\mathcal T_H$ of $\Omega$, the
mixed finite element method is to solve
(\ref{eq:mixweak1})-(\ref{eq:mixweak2}) in a pair of
finite-dimensional spaces $\boldsymbol \Sigma _H\subset \boldsymbol
\Sigma$ and $U_H\subset U$. That is, given an $f\in L^2(\Omega)$, to
find $(\boldsymbol \sigma _H,u_H)\in \boldsymbol \Sigma _H\times U_H$
such that
\begin{align}
  (\boldsymbol \sigma _H, \boldsymbol \tau _H)-({\rm div\,}\boldsymbol \tau _H, u _H)&= 0 & \forall \, \boldsymbol \tau_H \in \boldsymbol \Sigma _H \label{eq:mixfem1}\\
  ({\rm div\,}\boldsymbol \sigma _H, v_H)&= (f_H,v_H) & \forall
  \, v_H\in U_H. \label{eq:mixfem2}
\end{align}
Hereafter $f_H$ denotes the $L^2(\Omega)$ projection of $f$ onto
$U_H$. Namely, $f_H\in U_H$ such that $(f_H, v_H)=(f, v_H), \; \forall
v_H\in U_H$.  The well-posedness of the discrete problem
(\ref{eq:mixfem1})-(\ref{eq:mixfem2}), unlike the standard finite
element method for the primary variational formulation, is
non-trivial. One sufficient condition to construct stable finite
element spaces is to ensure the inf-sup condition still holds for the
discrete problem. Since 1970's many stable finite element spaces have
been introduced for this case, such as those of Raviart-Thomas spaces
\cite{Raviart.P;Thomas.J1977} and Brezzi-Douglas-Marini spaces
\cite{Brezzi.F;Douglas.J;Marini.L1985}. Recently it has been shown
that such stable finite element spaces can be constructed in an
elegant way using differential complex theory
\cite{Bossavit.A1988,Hiptmair.R1999a,Arnold.D2004,Arnold.D;Falk.R;Winther.R2005}.

The Raviart-Thomas spaces \cite{Raviart.P;Thomas.J1977} are defined
for an integer $p\geq 0$ by
\begin{align*}
  RT_H &= \bs \Sigma _H^p \times U_H^p, \; \text{where}\\
  \boldsymbol \Sigma _H^p(\mathcal T_H)&:=\{\boldsymbol \tau \in
  H({\rm div\,}; \Omega)\, :\, \boldsymbol \tau |_{T}\in {\boldsymbol
    P}_p(T)+\boldsymbol x\mathcal P_p(T), \; \forall \, T\in \mathcal
  T_H\},\\
  \text{and}\quad U_H^p(\mathcal T_H)&:=\{v\in L^2(\Omega)\,:\,
  v|_{T}\in \mathcal P_p(T),\; \forall \, T\in \mathcal T_H\},
\end{align*}
and where $\mathcal P_p(T)$ denotes the space of polynomials on $T$ of
degree at most $p$. 

The Brezzi-Douglas-Marini
spaces~\cite{Brezzi.F;Douglas.J;Marini.L1985} are defined for an
integer $p \geq 1$ by
\begin{align*}
  BDM_H &= \bs \Sigma _H^p \times U_H^p, \; \text{where}\\
  \boldsymbol \Sigma _H^p(\mathcal T_H)&:=\{\boldsymbol \tau \in
  H({\rm div\,}; \Omega)\, :\, \boldsymbol \tau |_{T}\in {\boldsymbol
    P}_p(T), \; \forall \, T\in \mathcal
  T_H\},\\
  \text{and}\quad U_H^p(\mathcal T_H)&:=\{v\in L^2(\Omega)\,:\,
  v|_{T}\in \mathcal P_{p-1}(T),\; \forall \, T\in \mathcal T_H\}.
\end{align*}

Since most results hold for both Raviart-Thomas and
Brezzi-Douglas-Marini spaces and $p$ is fixed in most places, we shall
use generic notation $(\bs \Sigma _H, U_H)$ to denote the pair in
$RT_H$ or $BDM_H$. The discrete problem posed on $(\boldsymbol \Sigma
_H, U_H)$ will satisfy the discrete inf-sup
condition~\cite{Brezzi.F;Fortin.M1991} from which the existence and
uniqueness of the finite element approximation $(\boldsymbol \sigma
_H,u_H)$ follows.

We shall use $\mathcal L$ and $\mathcal L_H$ to denote the
differential operators corresponding to
(\ref{eq:mixweak1})-(\ref{eq:mixweak2}) and
(\ref{eq:mixfem1})-(\ref{eq:mixfem2}), respectively. Those equations
can be formally written as
$$
\mathcal L(\boldsymbol \sigma,u)=f\quad \hbox{ and } \quad \mathcal
L_H(\boldsymbol \sigma _H, u_H)=f_H.
$$
We shall use the notation $(\bs \sigma, u)=\mcal L^{-1}f$ and $(\bs
\sigma_H, u_H)=\mcal L^{-1}_Hf_H$ to emphasis the dependence of
$f$. With an abuse of notation, we also use $\boldsymbol \sigma
=\mathcal L^{-1}f$ and $\boldsymbol \sigma _H=\mathcal L^{-1}_Hf_H$
when $\bs \sigma$ and $\bs \sigma _H$ are of interest.

\subsection{Adaptive methods through local refinement}
Let $\bs \sigma = \mcal L^{-1}f$ and $\bs \sigma _H= \mcal
L^{-1}_Hf_H$. We are mostly interested in the control of the error
$\|\bs \sigma - \bs \sigma _H\|$ which is usually more important than
control the error of scalar variable $u$ in mixed finite element
methods. Although the natural norm for the error is
$\|\bs \sigma - \bs \sigma _H\|_{H(\div)}$, we comment that, by (\ref{eq:mixweak2}) and
(\ref{eq:mixfem2}), $\|\div \bs \sigma - \div \bs \sigma
_H\|=\|f-f_H\|$ can be approximated efficiently without solving
equations and also may dominate the error $\|\bs \sigma -\bs \sigma
_H\|_{H(\div)}$; see Remark 3.4 in~\cite{Lovadina.C;Stenberg.R2006}.

The rate of the error $\|\boldsymbol \sigma -\boldsymbol \sigma _H\|$
for $\bs \sigma _H\in \bs \Sigma ^p_H(\mcal T_H)$ depends on the
regularity of the function being approximated and the regularity of
the mesh. If $\boldsymbol \sigma \in \boldsymbol H^{p+1}(\Omega)$ and
$\mathcal T_H$ is quasi-uniform with mesh size $H=\max _{T\in \mathcal
  T_H}{\rm diam}(T)$, then the following convergence result of optimal
order is well known \cite{Brezzi.F;Fortin.M1991}
\begin{equation}\label{eq:optorder}
\|\boldsymbol \sigma -\boldsymbol \sigma _H\|\leq C H^{p+1}\|\boldsymbol \sigma\|_{p+1}.
\end{equation}
The regularity result $\boldsymbol \sigma \in \boldsymbol
H^{p+1}(\Omega)$, however, may not be true in many applications,
especially for concave domains $\Omega$. Thus we cannot expect the
convergence result (\ref{eq:optorder}) on quasi-uniform grids in
general.

To improve the convergence rate, element sizes are adapted according
to the behavior of the solution. In this case, the element size in
areas of the domain where the solution is smooth can stay bounded well
away from zero, and thus the global element size is not a good measure
of the approximation rate. For this reason, when the optimality of the
convergence rate is concerned, $\# \mathcal T$, the number of
elements, is used to measure the approximation rate in the setting of
adaptive methods that involve local refinement.

We now briefly review the standard adaptive procedure. Given an
initial triangulation $\mathcal T_0$, we shall generate a sequence of
nested conforming triangulations $\mathcal T_k$ using the following
loop
\begin{equation}\label{keyloop}
  \hbox {\bf \small SOLVE } \rightarrow \hbox{ \bf \small ESTIMATE} \rightarrow \hbox{ \bf \small MARK } \rightarrow \hbox{ \bf \small REFINE}.
\end{equation}
More precisely, to get $\mathcal T_{k+1}$ from $\mathcal T_k$ we first
solve (\ref{eq:mixfem1})-(\ref{eq:mixfem2}) to get $\boldsymbol \sigma
_k$ on $\mathcal T_k$. The error is estimated using $\boldsymbol
\sigma_k$ and data. And the error estimator is used to mark a set of
of triangles or edges that are to be refined. Triangles are then
refined in such a way that the triangulation is still shape regular
and conforming in the sense of \cite{Ciarlet.P1978}.

We shall not discuss the step {\bf \small SOLVE} which deserves a
separate investigation. We assume that the solutions of the
finite-dimensional problems can be generated to any accuracy to
accomplish this in optimal space and time complexity. Multigrid-like
methods for mixed finite element methods on quasi-uniform grids can be
found in
\cite{Braess.D;Verfurth.R1990,Bramble.J;Pasciak.J;Xu.J1988,Brenner.S1992,Brenner.S1996,Gopalakrishnan.J2003,Oswald.P1997,Rusten.T;Vassilevski.P;Winther.R1996}.

The {\it a posteriori} error estimators are essential part of the {\bf
  \small ESTIMATE} step. Given a shape regular triangulation $\mcal
T_H$, let $\mcal E_H$ denote the edges of $\mcal T_H$. In this paper,
we shall use edge-wise error estimator $\eta _E$ for each edge $E\in
\mcal E_H$. See Section 4 for details.

The local error estimator $\eta _E$ is employed to mark for refinement
the elements whose error estimator is large. The way we mark these
triangles influences the efficiency of the adaptive algorithm.  In the
{\bf \small MARK} step we shall always use the marking strategy
firstly proposed by D\"orfler \cite{Dorfler.W1996} in order to prove
the convergence and the optimality of the local refinement strategy.

In the {\bf \small REFINE} step we need to carefully choose the rule
for dividing the marked triangles such that the mesh obtained by this
dividing rule is still conforming and shape regular. Such refinement
rules include red and green refinement
\cite{Bank.R;Sherman.A;Weiser.A1983}, longest edge refinement
\cite{Rivara.M1984,Rivara.M1984a}, and newest vertex bisection
\cite{Sewell.E1972,Mitchell.W1989,Mitchell.W1992}. Note that not only marked triangles get refined
but also additional triangles are refined to recovery the conformity
of triangulations. We would like to control the number of
elements added to ensure the overall optimality of the refinement
procedure. To this end, we shall use newest vertex bisection in this
article. We refer to
\cite{Mitchell.W1989,Verfurth.R1996,Binev.P;Dahmen.W;DeVore.R2004,Chen.L2006a}
for details of newest vertex bisection and only list two important
properties below.

Let $\mathcal T_k$ be a conforming triangulation refined from a shape
regular triangulation $\mathcal T_0$ using the new vertex bisection and let
$\mathcal M$ be the collection of all marked triangles going from
$\mathcal T_0$ to $\mathcal T_k$. Then
\begin{enumerate}
\item $\{\mathcal T_k\}$ is shape regular and
  the shape regularity only depends on $\mathcal
  T_0$;
\item $ \# \mathcal T_k \le \# \mathcal T_0 + C\# \mathcal M.$
\end{enumerate} 

Recently Stevenson \cite{Stevenson.R2008} showed that such results can
be extended to bisection algorithms of $n$-simplices. The optimality
of the adaptive finite element method in this paper, thus, could be extended to general space dimensions. 

\subsection{Approximation of the data}
We shall introduce the concept of data oscillation which is firstly introduced in \cite{Morin.P;Nochetto.R;Siebert.K2000}, and use it here for
the approximation of data. Such quantity measures intrinsic
information missing in the averaging process associated with finite
elements, which fails to detect fine structures of $f$.

For a set $A$, $H_A$ denotes the diameter of $A$. To simplify the
notation, we may drop the subscript if it is clear from the
context. For a triangulation $\mathcal T_H$ of $\Omega$ and a function
$f\in L^2(\Omega)$, we define a triangulation dependent norm
$$
\|H\,f\|_{0,\mathcal T_H}:=\Big (\sum _{T\in \mathcal
  T_H}H_T^2\|f\|_{0,T}^2\Big )^{1/2}.
$$ 

\begin{definition}
  Given a shape regular triangulation $\mathcal T_H$ of $\Omega$ and
  an $f\in L^2(\Omega)$, we define the data oscillation
$$
{\rm osc}(f,\mathcal T_H):=\|H(f-f_H)\|_{0,\mathcal T_H}.
$$
\end{definition}

Let $\mathcal P_N$ denote the set of triangulations constructed from
an initial triangulation $\mcal T_0$ by the newest vertex bisection
method with at most $N$ triangles. We define
$$
\|f\|_{\mathcal A^{s}_o}=\sup _{N\geq N_0}\Big (N^{s}\inf _{\mathcal T\in \mathcal P_N}{\rm osc}(f,\mathcal T)\Big ),
$$
where $N_0$ is a fixed integer representing the number of triangles in
$\mcal T_0$. We will recall a result of Binev, Dahmen and DeVore
\cite{Binev.P;Dahmen.W;DeVore.R2004} which shows that the
approximation of data can be done in an optimal way. The proof can be
found at \cite{Binev.P;Dahmen.W;DeVore.R2004}; See also
\cite{Binev.P;DeVore.R2004}.

\begin{theorem}[Binev, Dahmen and DeVore]\label{th:BDD}
  Given a tolerance $\varepsilon$, an $f\in L^2(\Omega)$ and a shape
  regular triangulation $\mathcal T_0$, there exists an algorithm
$$
\mathcal T_H=\hbox{\bf \small APPROX}(f, \mathcal T_0,\varepsilon)
$$ 
such that 
$$
{\rm osc}(f,\mathcal T_H)\leq \varepsilon,
\quad \text{and}\quad
\# \mathcal T_H-\#\mathcal T_0\leq C\|f\|_{\mathcal
  A^{1/s}_o}^{1/s}\varepsilon ^{-1/s}.
$$
\end{theorem}

\section{Quasi-orthogonality}
Unlike the primal formulation of Poisson equation, $\boldsymbol \sigma
_H$ is not the $L^2$-orthogonal projection of $\boldsymbol \sigma $ from
$\boldsymbol \Sigma $ to $\boldsymbol \Sigma _H$. Indeed the solution
$(\bs \sigma, u)$ of (\ref{eq:mixweak1})-(\ref{eq:mixweak2}) is the
saddle point of the following energy
$$
E(\bs \tau, v) = \frac{1}{2}\|\bs \tau \|^2 + (\div \tau, v) - (f, v),
\quad \tau \in \bs H(\div; \Omega), \, v\in L^2(\Omega).
$$
Namely
$$
E(\bs \sigma, u) = \inf _{\bs \sigma \in \bs H(\div; \Omega)} \sup _{v\in L^2(\Omega)} E(\bs \tau, v). 
$$
Similar result holds for the discrete solutions $(\bs \sigma _H,
u_H)$. The lack of orthogonality is the main difficulty which
complicates the convergence analysis of mixed finite element methods.

We shall use the fact the error $\boldsymbol \sigma
-\boldsymbol \sigma _H$ is orthogonal to the divergence free subspace
of $\boldsymbol \Sigma _H$ to prove a quasi-orthogonality result. In the
sequel we shall consider two conforming triangulations $\mathcal T_h$
and $\mathcal T_H$ which are nested in the sense that $\mathcal T_h$
is a refinement of $\mathcal T_H$. Therefore the finite element space
are nested i.e. $(\boldsymbol \Sigma _H, U_H)\subset (\boldsymbol
\Sigma _h, U_h)$.

\begin{lemma}\label{lm:orth}
  Given an $f\in L^2(\Omega)$ and two nested triangulation $\mcal T_h$
  and $\mcal T_H$, let $$(\boldsymbol \sigma,u) = \mathcal
  L^{-1}f,(\boldsymbol \sigma _h, u_h)=\mathcal L^{-1}_hf_h,
  (\tilde{\boldsymbol \sigma} _h, \tilde{u}_h)=\mathcal L^{-1}_hf_H,\,  
  \text{and}\, (\boldsymbol \sigma _H, u_H)=\mathcal L^{-1}_Hf_H.$$ Then
  \begin{equation}
    \label{eq:orth}
    (\bs \sigma - \bs \sigma _h, \tilde{\bs \sigma}_h-\bs \sigma _H)=0.
  \end{equation}
\end{lemma}
\begin{proof}
  Since $\tilde{\bs \sigma}_h-\bs \sigma _H\in \Sigma _h$, by
  (\ref{eq:mixfem1})-(\ref{eq:mixfem2}), we have
$$
(\boldsymbol \sigma -\boldsymbol \sigma _h, \tilde{\boldsymbol \sigma
  }_h-\boldsymbol \sigma _H)=(u-u_h,{\rm div\,} (\tilde{\boldsymbol
  \sigma}_h-\boldsymbol \sigma _H))=(u-u_h, f_H-f_H)=0.
$$ 
\end{proof}

To prove quasi-orthogonality, we need the following discrete stability result
\begin{equation}\label{eq:disstablity}
  \|\boldsymbol \sigma _h-\tilde {\boldsymbol \sigma} _h\|\leq \sqrt{C_0}\,{\rm osc}(f_h, \mathcal T_H),
\end{equation}
where the constant $C_0$ depends only on the shape regularity of 
$\mathcal T_H$. We shall leave the proof of (\ref{eq:disstablity}) to the next section and use it
 to derive the quasi-orthogonality result. 

\begin{theorem}\label{th:quasiorth}
  Given an $f\in L^2(\Omega)$ and two nested triangulations $\mcal T_h$
  and $\mcal T_H$, let $\boldsymbol \sigma =\mathcal
  L^{-1}f$, $\boldsymbol \sigma _h=\mathcal L^{-1}_hf_h$, and
  $\boldsymbol \sigma _H=\mathcal L^{-1}_Hf_H$.  Then
\begin{equation}\label{eq:quasiorth1}
  (\boldsymbol \sigma -\boldsymbol \sigma _h,\boldsymbol \sigma _h-\boldsymbol \sigma _H)\leq \sqrt{C_0}\,\|\boldsymbol \sigma -\boldsymbol \sigma _h\|{\rm osc}(f_h, \mathcal T_H),
\end{equation}
Thus for any $\delta >0$,
\begin{equation}\label{eq:quasiorth2}
  (1-\delta)\|\boldsymbol \sigma -\boldsymbol \sigma _h\|^2 \leq \|\boldsymbol \sigma -\boldsymbol \sigma _H\|^2 - \|\boldsymbol \sigma _h-\boldsymbol \sigma _H\|^2+\frac{C_0}{\delta}{\rm osc}^2(f_h,\mathcal T_H),
\end{equation}
and in particular when ${\rm osc}(f_h,\mathcal T_H)=0$, 
\begin{equation}
  \label{eq:orth1}
  \|\boldsymbol \sigma -\boldsymbol \sigma _h\|^2 = \|\boldsymbol \sigma -\boldsymbol \sigma _H\|^2 - \|\boldsymbol \sigma _h-\boldsymbol \sigma _H\|^2.
\end{equation}
\end{theorem}
\begin{proof}
  Let us introduce an intermediate solution $\tilde {\boldsymbol
    \sigma} _h=\mathcal L^{-1}_hf_H.$ By Lemma \ref{lm:orth},
  $(\boldsymbol \sigma -\boldsymbol \sigma _h,\tilde {\boldsymbol
    \sigma} _h-\boldsymbol \sigma _H)=0.  $ Thus
$$
(\boldsymbol \sigma -\boldsymbol \sigma _h,\boldsymbol \sigma
_h-\boldsymbol \sigma _H)=(\boldsymbol \sigma -\boldsymbol \sigma
_h,\boldsymbol \sigma _h-\tilde {\boldsymbol \sigma} _h)\leq
\|\boldsymbol \sigma -\boldsymbol \sigma _h\|\|\boldsymbol \sigma
_h-\tilde {\boldsymbol \sigma} _h\|.
$$
(\ref{eq:quasiorth1}) then follows from the inequality
(\ref{eq:disstablity}).

By the trivial identity $\bs \sigma - \bs \sigma _H = \bs \sigma - \bs
\sigma _h + \bs \sigma _h - \bs \sigma _H$, we have
$$
\|\boldsymbol \sigma-\boldsymbol \sigma _H\|^2=\|\boldsymbol \sigma
-\boldsymbol \sigma _h\|^2+\|\boldsymbol \sigma _h-\boldsymbol \sigma
_H\|^2+2(\boldsymbol \sigma -\boldsymbol \sigma _h,\boldsymbol \sigma
_h-\boldsymbol \sigma _H)
$$
When ${\rm osc}(f_h,\mcal T_H)=0$, by (\ref{eq:quasiorth1}),
$(\boldsymbol \sigma -\boldsymbol \sigma _h,\boldsymbol \sigma
_h-\boldsymbol \sigma _H)=0$ and thus (\ref{eq:orth1}) follows. In
general, we use
\begin{align*}
  \|\boldsymbol \sigma-\boldsymbol \sigma _H\|^2&=\|\boldsymbol \sigma -\boldsymbol \sigma _h\|^2+\|\boldsymbol \sigma _h-\boldsymbol \sigma _H\|^2+2(\boldsymbol \sigma -\boldsymbol \sigma _h,\boldsymbol \sigma _h-\boldsymbol \sigma _H)\\
  &\geq \|\boldsymbol \sigma -\boldsymbol \sigma _h\|^2+\|\boldsymbol \sigma _h-\boldsymbol \sigma _H\|^2-2\sqrt{C_0}\|\boldsymbol \sigma -\boldsymbol \sigma _h\|{\rm osc}(f,\mathcal T_H)\\
  &\geq \|\boldsymbol \sigma _h -\boldsymbol \sigma
  _H\|^2+(1-\delta)\|\boldsymbol \sigma -\boldsymbol \sigma
  _h\|^2-\frac{C_0}{\delta}{\rm osc}^2(f_h,\mathcal T_H),
\end{align*}
to prove (\ref{eq:quasiorth2}). In the last step, we have used the
inequality
$$
2ab\leq \delta a^2+\frac{1}{\delta}b^2, \quad \text{ for any }\, \delta >0.
$$
\end{proof}

A similar quasi-orthogonality result was obtained by Carstensen and
Hoppe \cite{Carstensen.C;Hoppe.R2006} for the lowest order
Raviart-Thomas spaces using a special relation to the non-conforming
finite element. Such relation for high order elements and
Brezzi-Douglas-Marini spaces are not easy to establish; see
\cite{Arnold.D;Brezzi.F1985} and
\cite{Cockburn.B;Gopalakrishnan.J2004,Cockburn.B;Gopalakrishnan.J2005a,Cockburn.B;Gopalakrishnan.J2005}
for discussion on this relation. In contrast the approach we used here
is more straight-forward.

\begin{remark}\label{remark} \rm 
The oscillation term ${\rm osc}(f_h,\mcal T_H)$ in (\ref{eq:quasiorth1}) and (\ref{eq:quasiorth2}) depends on both $\mcal T_h$ and $\mcal T_H$.  It can be changed to the quantity ${\rm osc}(f,\mcal T_H)$ which only depends on $\mcal T_H$. Indeed for each $T\in \mcal T_H$, we have
$$
\|f_h-f_H\|_{0,T}=\|Q_h(I-Q_H)f\|_{0,T}\leq \|f-f_H\|_{0,T},
$$
and thus ${\rm osc}(f_h,\mcal T_H)\leq {\rm osc}(f,\mcal T_H)$. This change is important for the construction of convergent AMFEM by showing the reduction of  ${\rm osc}(f,\mcal T_H)$.
\end{remark}

\section{Discrete stability for perturbation of data}
In this section, we shall prove the discrete stability result. We begin with a stability result in the continuous case. Let $u\in
H_0^1(\Omega)$ be the solution of the primal weak formulation
(\ref{primary}) of Poisson equation. Then $(-\nabla u, u)$ is the
solution to the dual weak formulation
(\ref{eq:mixweak1})-(\ref{eq:mixweak2}). The stability
result $\|\nabla u\| \leq \|f\|_{-1}$ is well-known in the
literature. The norm $\|f\|_{-1}$, however, is not easy to
compute. Instead we shall make use of the oscillation of data to bound
it.
\begin{theorem}\label{th:stabilityosc}
  Given a shape regular triangulation $\mathcal T_H$ of $\Omega$ and
  $f\in L^2(\Omega)$, let $(\bs \sigma, u)=\mathcal L^{-1}f$ and
  $(\tilde {\bs \sigma}, \tilde{u})=\mathcal L^{-1}f_H$,
  respectively. Then there exists a constant $C_0$ depending only on the
  shape regularity of $\mathcal T_H$ such that
  \begin{equation}
    \label{eq:stability}
    \|\boldsymbol \sigma -\tilde {\boldsymbol \sigma}\|\leq \sqrt{C_0}\, {\rm osc} (f,\mathcal T_H).
  \end{equation}
\end{theorem}
\begin{proof}
By (\ref{eq:mixweak1}) and (\ref{eq:mixweak2}), we have
$$
\|\bs \sigma - \tilde{\bs \sigma}\|^2=(\bs \sigma - \tilde{\bs
  \sigma},\bs \sigma - \tilde{\bs \sigma}) = (\div (\bs \sigma -
\tilde{\bs \sigma}), u-\tilde{u})=(f-f_H, u-\tilde{u}).
$$
Let $v$ be the solution of primal weak formulation of Poisson
equation with data $f-f_H$. Then $v=u-\tilde{u}$ and $-\nabla v=\bs
\sigma -\tilde{\bs \sigma}$. Recall that $Q_H: L^2(\Omega) \to U_H$ is the
$L^2$ projection into discontinuous polynomial spaces. So for each triangle $T\in \mcal T_H$, $(f-f_H, v_H)_{T}=0$ for any $v_H\in \mcal P_p(T)$. Therefore
\begin{align*}
  \|\bs \sigma - \tilde{\bs \sigma}\|^2&= (f-f_H,v)\\
  &=\sum _{T \in \mathcal T_H}(f-f_H,v-Q_Hv)_T\\
  &\leq \sqrt{C_0}\sum _{T \in \mathcal T_H}\|H(f-f_H)\|_{0,T}\|\nabla v\|_{0,T}\\
  &\leq \sqrt{C_0}\Big (\sum _{T \in \mathcal T_H}\|H(f-f_H)\|_{0,T}^2\Big
  )^{1/2}\|\bs \sigma - \tilde{\bs \sigma}\|.
\end{align*}
In the second step, we have used the error estimate 
$$
\|v-Q_Hv\|_{0,T} \leq \sqrt{C_0} H_T\|\nabla v\|_{0,T},
$$
which can be easily proved by Bramble-Hilbert lemma and the scaling
argument. The constant $C_0$ only depends on the shape regularity of
$\mathcal T_H$. The desired result then follows by canceling one
$\|\bs \sigma - \tilde{\bs \sigma}\|$.
\end{proof}

In the proof of Theorem \ref{th:stabilityosc}, we use the local error
estimate
$$
\|u-Q_Hu\|_{0,T}\leq \sqrt{C_0} H_T \|\nabla u\|_{0,T}=\sqrt{C_0} H_T\|\bs \sigma\|_T,
$$
for $u\in H_0^1(\Omega)$ and $\bs \sigma = -\nabla u$. The main
difficulty in the discrete case is that $u_h\in U_h \nsubseteq
H_0^1(\Omega)$. However we still have a similar localized error
estimate for $u_h-Q_Hu_h$.

\begin{lemma}\label{lm:QhQH}
  Let $\mcal T_h$ and $\mcal T_H$ be two nested triangulations, and
  let $(\bs \sigma _h, u_h)=\mcal L^{-1}_hf_h$.  Then for any $T\in
  \mcal T_H$, we have
  \begin{equation}
    \label{eq:QhQH}
    \|u_h-Q_Hu_h\|_{0,T} \leq \sqrt{C_0} H_T\|\bs \sigma _h\|_{0,T}.
  \end{equation}
\end{lemma}
The proof of this lemma is technical and postponed to the end of this
section. We use it to prove the following theorem.

\begin{theorem}\label{th:discretestability} 
  Let $\mathcal T_h$ and $\mathcal T_H$ be two nested conforming
  triangulations. Let $\tilde {\boldsymbol \sigma} _h=\mathcal
  L^{-1}_hf_H$ and $\boldsymbol \sigma _h=\mathcal L^{-1}_hf_h$. Then
  there exists a constant $C_0$, depending only on the shape
  regularity of $\mathcal T_H$ such that
\begin{equation}\label{eq:disstablity2}
  \|\boldsymbol \sigma _h-\tilde {\boldsymbol \sigma} _h\|\leq \sqrt{C_0}\,{\rm osc}(f_h, \mathcal T_H).
\end{equation}
\end{theorem}
\begin{proof}
Recall
that $\bs\sigma _h-\tilde{\bs\sigma}_h$ satisfies the equation
\begin{align}
  (\bs\sigma _h-\tilde{\bs\sigma}_h, \bs\tau_h) & =(u_h-\tilde u_h,
  \div
  \bs\tau _h), \quad &\forall \bs\tau _h \in \bs\Sigma _h \label{eq:diff1}\\
  (\div (\bs\sigma _h-\tilde{\bs\sigma}_h),v_h) &= (f_h-f_H,v_h),
  \quad &\forall v_h\in U_h \label{eq:diff2}.
\end{align}
We then choose $\bs\tau _h= \bs\sigma _h-\tilde{\bs\sigma}_h$ in
(\ref{eq:diff1}) and $v_h=u_h-\tilde u_h$ in (\ref{eq:diff2}) to
obtain
$$
\|\bs\sigma _h-\tilde{\bs\sigma}_h\|^2 = (u_h-\tilde u_h,
\div (\bs\sigma _h-\tilde{\bs\sigma}_h))
 = (v_h, f_h-f_H) = (v_h-Q_Hv_h, f_h-f_H).
$$
In the third step, we use the fact $f_H=Q_Hf=Q_Hf_h$ since $\mcal T_h$
and $\mcal T_H$ are nested. Thanks to (\ref{eq:QhQH}), we have
\begin{align*}
  \|\bs\sigma _h-\tilde{\bs\sigma}_h\|^2 & = \sum _{T\in \mcal
    T_H}(v_h-Q_Hv_h, f_h-f_H)_T\\
  &\leq \sqrt{C_0}\sum _{T\in \mcal T_H} H_T\|f_h-f_H\|_{0,T}\|\bs\sigma _h-\tilde{\bs\sigma}_h\|_{0,T} \\
  & \leq \sqrt{C_0}\left (\sum _{T\in \mcal T_H} H_T^2\|f_h-f_H\|_T^2\right
  )^{1/2} \|\bs\sigma _h-\tilde{\bs\sigma}_h\|.
\end{align*}
Canceling one $\|\bs\sigma _h-\tilde{\bs\sigma}_h\|$, we get the
desired result. 
\end{proof}

In the rest of this section, we shall prove Lemma \ref{lm:QhQH}. It is a modification of arguments in \cite{Arnold.D;Falk.R;Winther.R2000} from quasi-uniform grids to adaptive grids. The
first ingredient is the existence of a continuous right inverse of the
divergence as an operator from $\bs H_0^1(\Omega)$ into the space
$L^2_0(\Omega):=\{v\in L^2(\Omega): \int _{\Omega} v =0.\}$
\begin{lemma}\label{lm:div}
  Given a function $f\in L^2_0(\Omega)$, there exists a function $\bs
  \tau \in \bs H_0^1(\Omega)$ such that $$\div \bs \tau = f\quad
  \text{ and }\quad \|\bs \tau\|_{1}\leq C\|f\|.$$
\end{lemma}
The proof of this lemma for smooth or convex domains $\Omega$ is
pretty easy. One can solve the Poisson equation with Neumann boundary
condition
$$
\Delta \phi = f \text{ in } \Omega, \quad \frac{\partial
  \phi}{\partial \bs n}=0\, \, \text{ on }\, \partial \Omega.
$$
The condition $f\in L^2_0(\Omega)$ ensures the existence of the
solution. Then we let $\bs \tau = \grad \phi$ and modify the tangent
component of $\bs \tau$ to be zero \cite{Brenner.S;Scott.L2002}. See
also \cite{Arnold.D;Scott.L;Vogelius.M1988,Duran.R;Muschietti.M2001}
for a detailed proof on non-convex and general Lipschitz domains.

The second ingredient is an interpolation operator $\Pi _h: \bs
H^1(\Omega) \to \bs \Sigma_h$ with the following nice properties.
\begin{lemma}\label{lm:interpolation}
  There exists an interpolation operator $\Pi _h: \bs H^1(T) \to
  \bs \Sigma_h$ such that
\begin{enumerate}
\item $Q_h \div \bs \tau = \div \Pi_h \bs \tau, \quad \forall \bs \tau
  \in \bs H^1(\Omega);$
 
\item there exists a constant $C$ depending only on the shape
  regularity of $\mcal T_h$ such that
$$
\|\bs \tau - \Pi _h \bs \tau \|_{T}\leq Ch_{T}\|\bs \tau\|_{1,T},
\quad \forall T\in \mcal T_h, \forall \bs \tau \in \bs H^1(\Omega);
$$

\item for any $T\in \mcal T_h$ if $\bs \tau \in H_0^1(T)$, then $\Pi
  _h \bs \tau |_{\partial T}=0$.
\end{enumerate}
\end{lemma}
For the detailed construction of such interpolation operator and proof
of these properties, we refer to \cite{Hiptmair.R1999a} and
\cite{Arnold.D;Falk.R;Winther.R2006}.

\smallskip

\noindent {\it Proof of Lemma \ref{lm:QhQH}}
We first note that $u_h-Q_Hu_h=(Q_h-Q_H)u_h$ since $Q_hu_h=u_h$. For
any $T\in \mcal T_H$, by the definition of $L^2$ projection $Q_H$,
we have, $\int _T(Q_h-Q_H)u_h=0$ i.e. $(Q_h-Q_H)u_h\in L_0^2(T)$.
We thus can apply Lemma \ref{lm:div} to find a function $\bs \tau \in
\bs H_0^1(T)$ such that
$$
\div \bs \tau = (Q_h-Q_H)u_h, \text{ in } T \quad \text{ and }\quad
\|\bs \tau \|_{1,T}\leq C\|(Q_h-Q_H)u_h\|_{0,T}.
$$ 
We extend $\bs \tau$ to $\bs H^1(\Omega)$ by zero. Note that
\begin{equation}
  \label{eq:pihpiH}
  (\Pi _h - \Pi_H) \bs \tau \in \bs \Sigma _h, \; \text{ and }\, {\rm supp} (\Pi _h - \Pi_H) \bs \tau \subseteq T.
\end{equation}
With such $\bs \tau$, we have
$$
\|(Q_h-Q_H)u_h\|_{0,T}^2 =((Q_h-Q_H)u_h, \div \bs \tau)_T =
(u_h,(Q_h-Q_H)\div \bs\tau)_T.
$$
Then using the commuting property (Lemma \ref{lm:interpolation} (1))
and the locality of $\bs \tau$, we have
$$
(u_h,(Q_h-Q_H)\div \bs\tau)_T=(u_h,(Q_h-Q_H)\div \bs\tau)_{\Omega} =
(u_h, \div(\Pi _h-\Pi _H)\bs \tau)_{\Omega}.
$$
Now we shall use the fact $(\sigma _h, u_h)$ is the solution of
(\ref{eq:mixweak1}) and (\ref{eq:mixweak2}) and, again, the locality
of $\bs \tau$ to get
$$
(u_h, \div(\Pi _h-\Pi _H)\bs \tau)_{\Omega} =(\bs\sigma_h,(\Pi
_h-\Pi_H)\bs \tau)_{\Omega} = (\bs\sigma_h,(\Pi _h-\Pi_H)\bs \tau)_{T}.
$$
Using the approximation property of $\Pi _h$ (Lemma \ref{lm:interpolation} (2)), we get
\begin{align*}
  (\bs\sigma_h,(\Pi _h-\Pi_H)\bs \tau)_{T}
  &\leq \|\bs\sigma _h\|_{0,T}\big (\|\bs\tau - \Pi _h
  \bs\tau\|_{0,T}+\|\bs\tau - \Pi _H \bs\tau\|_{0,T})\\
  &\leq CH_T\|\bs\sigma _h\|_{0,T}\|\bs\tau\|_{1,T}.
\end{align*}
So we have
$$
\|(Q_h-Q_H)u_h\|_{0,T}^2 \leq CH_T\|\bs\sigma
_h\|_{0,T}\|\bs\tau\|_{1,T}\leq CH_T\|\bs\sigma
_h\|_T\|(Q_h-Q_H)u_h\|_{0,T}.
$$
Canceling one $\|(Q_h-Q_H)u_h\|_{T}$, we obtain the desired
result. \qed

\section{A Posterior Error Estimate for Mixed Finite Element Methods}
In this section we shall follow Alonso \cite{Alonso.A1996} to present
{\it a posteriori} error estimate for mixed finite element
methods. Other \textit{a posteriori} error estimators for the mixed
finite element methods can be found at
\cite{Carstensen.C1997,Wohlmuth.B;Hoppe.R1999,Hoppe.R;Wohlmuth.B1997,Gatica.G;Maischak.M2004,Larson.M;Maqvist.A2005,Lovadina.C;Stenberg.R2006}. Our
analysis could be adapted to these error estimators also.

\subsection{A posteriori error estimator and existing results}
\label{sec:def}
Let us begin with the definition of the error estimator. For any edge
$E\in \mathcal E_H$, we shall fix an unit tangent vector $\bs t_E$ for
$E$. We denote the patch of $E$ consisting of triangles sharing $E$ by
$\Omega _E$.

\begin{definition}
  Given a triangulation $\mathcal T_H$, for an $E\in \mathcal E_H$ and
  $E\notin \partial \Omega$, let $\Omega_E=T \cup \tilde T$. For any
  $\bs \sigma _H\in \bs \Sigma_H$, we define the jump of $\boldsymbol
  \sigma _H$ across edge $E$ as
\begin{equation}\label{eq:JE}
  J_E(\boldsymbol \sigma _H)=\big [\boldsymbol \sigma _H\cdot \boldsymbol t_E \big ]:=\boldsymbol \sigma _H|_{T}\cdot \boldsymbol t_E- \boldsymbol \sigma _H|_{\tilde T}\cdot \boldsymbol t_E.
\end{equation}
If $E\in \mathcal E_H\cap \partial \Omega$, we define $J_E(\boldsymbol
\sigma _H)=\boldsymbol \sigma _H\cdot \boldsymbol t_E$. The edge error
estimator is defined as
$$
\eta _E^2(\boldsymbol \sigma _H)= \|H\, {\rm rot}\,
\boldsymbol \sigma _H\|^2_{0,\Omega _E}+\|H^{1/2}J_E(\boldsymbol
\sigma _H)\|_{0,E}^2.
$$
For a subset $\mathcal F_H\subseteq \mathcal E_H$, we define
$$
\eta ^2(\boldsymbol \sigma _H, \mathcal F_H)
:=
\sum _{E\in \mathcal F_H}\eta _E^2(\boldsymbol \sigma _H).
$$
\end{definition}

The error estimator $\eta _E(\bs \sigma _H)$ is continuous with
respect to $\bs \sigma _H$ in $L^2$-norm. Namely we have the following
inequality.
\begin{lemma}
  Given an $f\in L^2(\Omega)$ and a shape regular triangulation $\mcal
  T_H$, let $\bs \sigma _H,\bs \tau _H \in \bs \Sigma _H$. There
  exists constant $\beta$ such that
  \begin{equation}
    \label{eq:continunity}
    \beta \left | \eta ^2(\bs \sigma _H, \mathcal E_H) -  \eta ^2(\bs \tau _H, \mathcal E_H)\right |\leq \|\bs \sigma _H -\bs \tau _H\|^2.
  \end{equation}
\end{lemma}
\begin{proof}
  It can be easily proved by the triangle inequality and inverse
  inequality.
\end{proof}

We shall recall Alonso's results below and prove a discrete upper bound
later. Since the data $f$ is not included in the definition of our error
estimator $\eta _E$, the upper bound contains an additional data oscillation term
which is different from the standard one in \cite{Verfurth.R1996}.

\begin{theorem}[Upper bound]\label{th:Alonso}
  Given an $f\in L^2(\Omega)$ and a shape regular triangulation $\mcal
  T_H$, let $\boldsymbol \sigma =\mathcal L^{-1}f$ and $\boldsymbol
  \sigma _H=\mathcal L^{-1}_Hf_H$. There exist constants $C_0$ and
  $C_1$ depending only on the shape regularity of $\mathcal T_H$ such
  that
  \begin{equation}
    \label{eq:upperbound}
  \|\boldsymbol \sigma -\boldsymbol \sigma _H\|^2 \leq C_1 \eta
  ^2(\boldsymbol
  \sigma _H, \mathcal E_H)+C_0{\rm osc}^2(f, \mathcal T_H).    
  \end{equation}
\end{theorem}

\begin{theorem}[Lower bound]
  Given an $f\in L^2(\Omega)$ and a shape regular triangulation $\mcal
  T_H$, let $\boldsymbol \sigma =\mathcal L^{-1}f$ and $\boldsymbol
  \sigma _H=\mathcal L^{-1}_Hf_H$. There exists constant $C_2$
  depending only on the shape regularity of $\mathcal T_H$ such that
  \begin{equation}
\label{eq:lowerbound} C_2\eta ^2(\boldsymbol \sigma _H, \mathcal
  E_H) \leq \|\boldsymbol \sigma -\boldsymbol \sigma _H\|^2,  
\end{equation}
for Raviart-Thomas spaces. 

For Brezzi-Douglas-Marini spaces, (\ref{eq:lowerbound}) holds when
${\rm osc}(f,\mcal T_H)=0$.
\end{theorem}

When ${\rm osc}(f,\mcal T_H)=0$, (\ref{eq:upperbound}) and
(\ref{eq:lowerbound}) implies that $C_2/C_1\leq 1$. This ratio is a
measure of the precision of the indicator.

\subsection{Discrete upper bound}
We shall give a discrete version of the upper bound
(\ref{eq:upperbound}). The main tool is the discrete Helmholtz
decomposition.

Given a shape regular triangulation $\mathcal T_h$, let
$$S^{p}_{h}=\{\psi_h\in \mathcal C(\overline \Omega): \psi _h|_{T}\in
\mathcal P_p(T), \; \forall \, T\in \mathcal T_h\}$$ denote the
standard continuous and piecewise polynomial finite element spaces of
$H^1(\Omega)$. To introduce the discrete Helmholtz decomposition, we
define the dual operator operator of ${\rm div}: \boldsymbol \Sigma _h
\mapsto U_h$.
\begin{definition}
 We define ${\rm grad} _h: U_h\mapsto (\boldsymbol \Sigma _h)^*$ by
 $$
 ({\rm grad} _h v_h,\boldsymbol \tau _h) = -(v_h, {\rm div\, }
 \boldsymbol \tau _h), \quad \forall \boldsymbol \tau _h\in
 \boldsymbol \Sigma _h.
$$
\end{definition}
We emphasis that ${\rm grad}_h$ is not simply the restriction of ${\rm
  grad}$ to $U_h$ since $U_h$ is not a subspace of $H^1(\Omega)$. The following discrete Helmholtz
decomposition is well known in the literature; See, for example,
\cite{Fix.G;Gunzburger.M;Nicolaides.R1981,Brandts.J1994a,
  Arnold.D;Falk.R;Winther.R2005,Bochev.P;Gunzburger.M2005}.

\begin{theorem}[Discrete Helmholtz Decomposition in $\mathbb R^2$]
  Given a triangulation $\mathcal T_h$, for $p$-th order
  Raviart-Thomas finite element spaces $(\boldsymbol \Sigma _h^p,
  U_h^p)$, we have the following orthogonal (with respect to $L^2$
  inner product) decomposition
$$
\Sigma _h^p={\rm curl}(S^{p+1}_{h})\oplus {\rm grad _h}\,  (U_h^p).
$$
For Brezzi-Douglas-Marini finite element spaces $(\boldsymbol \Sigma
_h^p, U_h^p)$, we have the following orthogonal (with respect
to $L^2$ inner product) decomposition
$$
\Sigma _h^p={\rm curl}(S^{p+1}_{h})\oplus {\rm grad _h}\,  (U_h^{p-1}).
$$
\end{theorem}

We are in the position to present a discrete version of the upper
bound.
\begin{theorem}\label{th:upperbound}
  Let $\mathcal T_h$ and $\mathcal T_H$ be two nested conforming
  triangulations. Let $\boldsymbol \sigma _h=\mathcal L^{-1}_hf_h$ and
  $\boldsymbol \sigma _H=\mathcal L^{-1}_Hf_H$, and let $\mathcal F _H
  = \{E\in \mcal E_H: E\notin \mathcal E_h\}$. Then there exist
  constants depending only on the shape regularity of $\mathcal T_H$ such
  that
  \begin{equation}\label{eq:upper}
    \|\boldsymbol \sigma _h-\boldsymbol \sigma _H\|^2\leq C_1\eta ^2(\boldsymbol \sigma _H, \mathcal F_H)+C_0{\rm osc}^2(f_h, \mathcal T_H)
      \end{equation}
and
\begin{equation}
  \label{eq:upperF}
\# \mathcal F_H\leq 3(\# \mathcal T_h -\# \mathcal T_H).
  \end{equation}
\end{theorem}
\begin{proof}
The inequality (\ref{eq:upperF}) follows from
$$
\# \mathcal F_H\leq \# \mcal E_h - \# \mcal E_H \leq 3(\# \mathcal T_h
-\# \mathcal T_H).
$$
To prove (\ref{eq:upper}), again we introduce the intermediate
solution $\tilde {\boldsymbol \sigma} _h=\mathcal L^{-1}_hf_H$. By the
discrete Helmholtz decomposition, we have
$$
\tilde {\boldsymbol \sigma} _h-\boldsymbol \sigma _H={\rm grad}_h\,
\phi _h + {\rm curl\,} \psi _h,
$$
where $\phi _h\in U_h^p, \psi _h\in S^{p+1}_h$ for Raviart-Thomas
spaces, and $\phi _h\in U_h^{p-1}, \psi _h\in S^{p+1}_h,$ for
Brezzi-Douglas-Marini spaces. The decomposition is $L^2$-orthogonal i.e.
\begin{equation}
  \label{eq:disL2}
\|\tilde{\bs\sigma}_h - \bs\sigma _H\|^2=\|\grad _h \phi
_h\|^2+\|\curl \psi _h\|^2.
\end{equation}

  In two dimensions, $\|\curl \psi _h\| = \|\grad \psi _h\|$ and thus
  (\ref{eq:disL2}) implies that
\begin{equation}
  \label{eq:psiH1}
  |\psi _h|_1 \leq \|\tilde{\bs\sigma}_h - \bs\sigma _H\|.
\end{equation}
Since
$$
(\tilde {\boldsymbol \sigma} _h-\boldsymbol \sigma _H, {\rm
  grad}_hv_h) =({\rm div\,}\,(\tilde {\boldsymbol \sigma}
_h-\boldsymbol \sigma _H), v_h) =(f_H - f_H, v_h-v_H)=0,
$$
we have
$$
\|\tilde{\bs\sigma}_h - \bs\sigma _H\|^2= (\tilde{\bs\sigma}_h -
\bs\sigma _H,\grad _h \phi _h) + (\tilde{\bs\sigma}_h - \bs\sigma _H,
\curl \psi _h)=(\tilde{\bs\sigma}_h - \bs\sigma _H,
\curl \psi _h).
$$
Since $\div \curl \psi =0$, $(\tilde{\bs\sigma}_h,\curl \psi
_h)=0$ and $(\bs \sigma _H, \curl \psi _H)=0$ for any $\psi _H\in
S^{p+1}_H$. 

Choosing $\psi _H = \mcal I_H\psi _h$ using some local
quasi-interpolation, for example the Scott-Zhang quasi-interpolation
\cite{Scott.R;Zhang.S1990}, $\mcal I_H: S^{p+1}_{h}\mapsto
S^{p+1}_{H}$ such that
$$
\|\psi _h-\mcal I_H\psi _h\|_{0,E}\leq C H^{1/2}_E|\psi
_h|_{1,\Omega_E}\quad \hbox{and} \quad \|\psi _h - \mcal I_H\psi
_h\|_{0,T}\leq C H_T|\psi _h|_{1,\Omega_T},
$$
where $\Omega _T=\{T_H\subset \mathcal T_H : T_H\cap T\neq
\emptyset\}$. Furthermore the quasi-interpolation $\mcal I_H$ is local
in the sense that if $T\in \mcal T_H\cap \mcal T_h$ or $E\in \mcal
E_H\cap \mcal E_h$ (i.e. $T$ or $E$ is not refined), then $(\psi _h -
\mcal I_H\psi _h)|_T=0$ or $(\psi _h - \mcal I_H\psi _h)|_E=0$,
respectively. With such choice of $\mathcal F_H$ and $\psi _H$, we
have
\begin{align*}
  \|\tilde {\boldsymbol \sigma} _h-\boldsymbol \sigma _H\|^2&=(\tilde {\boldsymbol \sigma} _h-\boldsymbol \sigma _H,{\rm curl\,} \psi _h)=(-\boldsymbol \sigma _H,{\rm curl\,} (\psi _h-\psi _H))\\
  &=\sum _{T\in \mathcal T_H}\Big [\sum _{E\in \partial T }
  (\boldsymbol \sigma _H\cdot \boldsymbol t_E,\psi _h-\psi _H)_E+({\rm
    rot}\, \boldsymbol \sigma _H, \psi _h - \psi _H)_T \Big ]\\
  &\leq \sum _{E\in \mcal E_H}[\bs \sigma _H]\|\psi _h - \psi
  _H\|_{0,E} + \sum _{T\in \mcal T_H}\|{\rm rot} \, \bs \sigma
  _H\|\|\psi _h - \psi _H\|_{0,T},\\
  &\leq C\eta(\boldsymbol \sigma _H,\mathcal F_H)|\psi _h|_1 \leq
  C_1\eta(\boldsymbol \sigma _H,\mathcal F_H)\|\tilde {\boldsymbol
    \sigma} _h-\boldsymbol \sigma _H\|.
\end{align*}
Canceling one $\|\tilde {\boldsymbol \sigma} _h-\boldsymbol \sigma
_H\|,$ we get 
\begin{equation}
  \label{eq:C1}
 \|\tilde {\boldsymbol \sigma} _h-\boldsymbol \sigma
_H\|\leq C\eta (\boldsymbol \sigma _H, \mathcal F_H). 
\end{equation}

Now we write $\boldsymbol \sigma _h-\boldsymbol \sigma _H=\boldsymbol \sigma _h-\tilde {\boldsymbol \sigma}
_h+\tilde {\boldsymbol \sigma}_h-\boldsymbol \sigma _H$ and note that 
$$(\boldsymbol \sigma _h-\tilde {\boldsymbol \sigma}
_h,\tilde {\boldsymbol \sigma} _h-\boldsymbol \sigma _H)=(u_h-\tilde
{u}_h, \div (\tilde {\boldsymbol \sigma} _h-\boldsymbol \sigma
_H)=(u_h-\tilde {u}_h, f_H-f_H)=0.
$$  
Combining (\ref{eq:C1}) and (\ref{eq:disstablity}), we then have
$$
\|\boldsymbol \sigma _h-\boldsymbol \sigma _H\|^2= \|\tilde
{\boldsymbol \sigma} _h-\boldsymbol \sigma _H\|^2+\|\boldsymbol \sigma
_h-\tilde {\boldsymbol \sigma} _h\|^2\leq C_1\eta^2(\boldsymbol \sigma
_H, \mathcal F_H)+C_0{\rm osc}^2(f_h, \mathcal T_H).
$$
\end{proof}

\section{Convergence and Optimality of AMFEM}
In this section we shall present our algorithms and prove their
convergence and optimality. It is adapted from the literature
\cite{Morin.P;Nochetto.R;Siebert.K2000,Stevenson.R2007,Dorfler.W1996,Morin.P;Nochetto.R;Siebert.K2002,Morin.P;Nochetto.R;Siebert.K2003}. For
the completeness we include them here and prove some important
technical results.

We first present our algorithms. It mainly follows from the algorithm
proposed in \cite{Morin.P;Nochetto.R;Siebert.K2003}. The difference is
that we do not impose an interior point
property in the refinement step. 

Let $\mathcal T_0$ be a initial shape regular triangulation, a right hand
side $f\in L^2(\Omega)$, a tolerance $\varepsilon$, and $0<\theta,
\tilde \theta, \mu <1$ three parameters. Thereafter we replace the
subscript $h$ by an iteration counter called $k$. For a marked edge
set $\mcal M_k$, we denote by $\Omega _{\mcal M_k}=\cup _{E\in \mcal
  M_k}\Omega _E$.

\medskip
\noindent $[\mathcal T_N,\bs \sigma_N]$={\bf \small AMFEM}$(\mathcal
T_0, f,\varepsilon, \theta, \tilde \theta,\mu)$

$\eta = \varepsilon, k=0$

\smallskip

{\bf \small WHILE} $\eta \geq \varepsilon$, {\bf \small DO}

\quad Solve (\ref{eq:mixfem1})-(\ref{eq:mixfem2}) on
$\mathcal T_k$ to get the solution $\bs \sigma_k$.

\quad Compute the error estimator
$\eta=\eta(\bs \sigma_k,\mathcal E_k)$.

\quad Mark the minimal edge set $\mathcal M_k$ such that
  \begin{equation}
    \label{eq:theta}
 \eta ^2(\bs \sigma_k,\mathcal M_k)\geq \theta \, \eta ^2(\bs
 \sigma_k,\mathcal E_k).    
  \end{equation}

  \quad If ${\rm osc}(f,\mcal T_k) > {\rm osc}(f,\mcal T_0) \mu ^k$,
  enlarge $\mcal M_k$ such that
  \begin{equation}
    \label{eq:tildetheta}
    {\rm osc}(f,\Omega_{\mathcal M_k}) \geq \tilde \theta \, {\rm osc}(f,\mcal T_k).    
  \end{equation}

  \quad Refine each triangle $\tau \in \Omega_{\mathcal
      M_k}$ by the newest vertex bisection to get $\mcal T_{k+1}$.

\quad $k=k+1$.

\smallskip

{\bf \small END WHILE}

\smallskip

$\mcal T_N=\mcal T_k$.

\noindent {\bf \small END AMFEM}

\subsection{Convergence of AMFEM}
We shall prove the algorithm {\bf \small AMFEM} will terminate in
finite steps by showing the reduction of the sum of the error and the
error estimator.

We first summarize the main ingredients in the following lemma with
the following short notation:
$$
  e_{k} = \|\bs \sigma - \bs \sigma _{k}\|^2, 
  E_k = \|\bs \sigma_{k+1} - \bs \sigma _k\|^2, 
  o_k =  \operatorname{osc}^2(f,\mathcal T_k),
\; \text{and} \; \eta _k = \eta ^2(\bs \sigma _k, \mathcal E_k).
$$

\begin{lemma}
One has the following inequalities
\begin{align}
  (1-\delta)\, e_{k+1} &\leq e_k - E_k + \frac{C_0}{\delta}\, o_k, \quad \text{for any }\delta >0\label{eq:1}\\
  \beta \,\eta _{k+1} &\leq \beta (1-\frac{1}{2}\theta)\,\eta _k +
  E_k, \label{eq:2}\\
  e_k &\leq C_1\,\eta _k + C_0\, o_k. \label{eq:3}
\end{align}
\end{lemma}
\begin{proof}
  (\ref{eq:1}) is the quasi-orthogonality (\ref{eq:quasiorth2})
  established in Theorem \ref{th:quasiorth} and Remark \ref{remark}. (\ref{eq:3}) is the
  upper bound (\ref{eq:upperbound}) in Theorem \ref{th:upperbound}. We
  only need to prove (\ref{eq:2}). By the continuity of the error
  estimator (\ref{eq:continunity}), we have
  \begin{equation}
    \label{eq:eta0}
\beta \eta^2(\bs \sigma _{k+1},\mcal E_{k+1}) \leq \beta \eta
^2(\bs\sigma_k, \mathcal E_{k+1}) + E_k.    
  \end{equation}
Let $\mathcal N_{k+1}=\mcal E_{k+1}\backslash \mcal E_k$ be the new
edges in $\mcal T_{k+1}$ and $\overline{\mcal M}_{k} \subseteq \mcal
E_k$ be the refined edge in $\mcal T_{k}$. It is obvious that
$\mathcal E_k \backslash \overline{\mathcal M}_k = \mathcal E_{k+1}
\backslash \mathcal N_{k+1}$. For an edge $E\in \mcal N_{k+1}$, if it
is an interior edge of some triangle $T\in \mcal T_k$, then $J_E(\bs
\sigma _k)=0$ since $\bs \sigma _k$ is a polynomial in $T$. For other edges, it is at least half of some edge in
$\overline{\mcal M}_{k}$ and thus
\begin{equation}
  \label{eq:eta1}
\eta ^2(\bs \sigma _k, \mathcal N_{k+1})\leq \frac{1}{2}\, \eta ^2(\bs
\sigma _k, \overline{\mathcal M}_k).  
\end{equation}
Since some edges are refined for the conformity of triangulation,
$\mcal M_k\subseteq \overline{\mcal M}_k$. By the marking strategy (\ref{eq:theta}), we have 
\begin{equation}
  \label{eq:eta2}
  \eta ^2(\bs \sigma_k,\overline{\mathcal M}_k) \geq  \eta ^2(\bs \sigma_k,\mathcal M_k)\geq \theta \, \eta ^2(\bs
  \sigma_k,\mathcal E_k).  
\end{equation}
Combining (\ref{eq:eta1}) and (\ref{eq:eta2}), we get
\begin{align*}
  \eta ^2(\bs \sigma _k, \mathcal E_{k+1}) &= \eta ^2(\bs \sigma _k,
  \mathcal N_{k+1}) + \eta ^2(\bs \sigma _k, \mathcal E_{k+1}
  \backslash
  \mathcal N_{k+1})\\
  &\leq \frac{1}{2} \eta ^2(\bs \sigma _k, \overline{\mathcal M}_k) +
  \eta ^2(\bs \sigma _k, \mathcal E_k
  \backslash \overline{\mathcal M}_k)\\
  &\leq -\frac{1}{2} \eta ^2(\bs \sigma _k, \overline{\mathcal M}_k) + \eta ^2(\bs \sigma _k, \mcal E_k)\\
  & \leq (1-\frac{1}{2}\theta)\eta ^2(\bs \sigma _k, \mcal E_k).
\end{align*}
Substituting to (\ref{eq:eta0}) we then get (\ref{eq:2}).
\end{proof}

\begin{theorem}
  When 
  \begin{equation}
    \label{eq:delta}
0<\delta < \min\{ \frac{\beta}{2C_1}\theta, 1\},    
  \end{equation}
there exists $\alpha \in (0,1)$ and $C_{\delta}$ such that
  \begin{equation}
    \label{eq:reduction}
    (1-\delta)e_{k+1} + \beta \eta _{k+1} \leq \alpha \big [(1-\delta)e_k + \beta \eta _k\big ] + C_{\delta}\, o_k.
  \end{equation}
\end{theorem}
\begin{proof}
First (\ref{eq:1}) + (\ref{eq:2}) gives
$$
(1-\delta)e_{k+1} + \beta \eta _{k+1} \leq e_k + \beta
(1-\frac{1}{2}\theta)\eta _k + \frac{C_0}{\delta} \, o_k.
$$
Then we separate $e_k$ and use (\ref{eq:2}) to bound
\begin{align*}
  e_k &= \alpha (1-\delta)e_k + [1-\alpha (1-\delta)]e_k \\
  &\leq \alpha (1-\delta)e_k + [1-\alpha (1-\delta)](C_1\eta_k +
  C_0\,o_k).
\end{align*}
Therefore we obtain
\begin{align*}
  (1-\delta)e_{k+1} + \beta \eta _{k+1} \leq \alpha \left
    \{(1-\delta)e_k + \frac{[1-\alpha
      (1-\delta)]C_1}{\alpha}\eta_k \right \} + C_{\delta}\,o_k.
\end{align*}
Now we choose $\alpha$ such that 
$$
\frac{[1-\alpha (1-\delta)]C_1}{\alpha} = \beta, 
$$
i.e.
$$
\alpha =
\frac{C_1+(1-\frac{1}{2}\theta)\beta}{C_1(1-\delta)+\beta}=\frac{C_1+
  \beta -\frac{1}{2}\theta\beta}{C_1+\beta-C_1\delta}.
$$
By the requirement of $\delta$ (\ref{eq:delta}), we conclude $\alpha
\in (0,1)$.
\end{proof}

\begin{theorem}\label{th:reduction}
  Let $\bs \sigma_k$ be the solution obtained in the $k$-th loop in
  the algorithm {\bf \small AMFEM}, then for any $0<\delta < \min\{ \frac{\beta}{2C_1}\theta, 1\},$ there exist positive
  constants $C_{\delta}$ and $0<\gamma _{\delta}<1$ depending only on given data and the
  initial grid such that, 
$$
(1-\delta)\|\bs \sigma - \bs \sigma _k\|^2 + \beta \eta ^2(\bs \sigma
_k, \mathcal T_k) \leq C_{\delta} \gamma^k_{\delta},
$$
and thus the algorithm {\bf \small AMFEM} will terminate in finite
steps.
\end{theorem}
\begin{proof}
  The proof is identical to that of Theorem 4.7 in
  \cite{Morin.P;Nochetto.R;Siebert.K2003} using (\ref{eq:reduction}).
\end{proof}

\subsection{Optimality of AMFEM}

Let $\mathcal T_0$ be an initial quasi-uniform triangulation with $\#
\mathcal T_0>2$ and $\mathcal P_N$ be the set of all triangulations
$\mathcal T$ which is refined from $\mathcal T_0$ and $\#\mathcal
T\leq N$. For a given triangulation $\mathcal T$, the solution of the
mixed finite element approximation of Poisson equation will be
denoted by $\boldsymbol \sigma_{\mathcal T}$. We define
$$
\mathcal A^s=\{\boldsymbol \sigma \in \boldsymbol \Sigma :
\|\boldsymbol \sigma\|_{\mathcal A^s}<\infty,\; \hbox{ with
}\|\boldsymbol \sigma\|_{\mathcal A^s}=\sup _{N\geq \#\mathcal
  T_0}\big (N^{s}\inf _{\mathcal T\in \mathcal P_N}\|\boldsymbol
\sigma-\boldsymbol \sigma_{\mathcal T}\|\big)\}.
$$ 
An adaptive finite element method realizes optimal convergence rates
if whenever $\boldsymbol \sigma \in \mathcal A^s$, it produces
approximation $\boldsymbol \sigma_N$ with respect to triangulations
$\mathcal T_N$ elements such that $ \|\boldsymbol \sigma-\boldsymbol
\sigma_N\|\leq C (\# \mathcal T_N)^{-s}.$

For simplicity, we consider the following algorithm which separates the
reduction of data oscillation and error.

\begin{enumerate}
\item $[\mathcal T_H, f_H]=\hbox{\bf \small APPROX }(f, \mathcal
  T_0,\varepsilon/2)$
\smallskip
\item $[\boldsymbol \sigma _N, \mathcal
  T_N]=\hbox{\bf AMFEM}(\mathcal T_H,f_H,\varepsilon/2,\theta,0,1)$
\end{enumerate}

The advantage of separating data error and discretization error is that in the second step,
data oscillation is always zero since the input data $f_H$ is
piecewise polynomial in the initial grid $\mcal T_H$ for {\bf \small
  AMFEM}. In this case, we also list all ingredients needed for the
optimality of adaptive procedure.

\begin{enumerate}
\item Orthogonality: 
$$\|\boldsymbol \sigma -\boldsymbol \sigma
  _{k+1}\|^2 = \|\boldsymbol \sigma -\boldsymbol \sigma _k\|^2 -
  \|\boldsymbol \sigma _{k+1}-\boldsymbol \sigma _k\|^2$$

\item Discrete upper bound: 
$$ \|\boldsymbol \sigma _{k+1}-\boldsymbol
  \sigma _k\|^2\leq C_1\eta ^2(\boldsymbol \sigma _k, \mathcal
  F_k)\; \text{ and }\;\# \mathcal F_k\leq
  3(\# \mathcal T_{k+1} -\# \mathcal T_k).$$

\item Lower bound: $$C_2\eta ^2(\boldsymbol \sigma _k, \mathcal
  E_k) \leq \|\boldsymbol \sigma -\boldsymbol \sigma _k\|^2.$$
\end{enumerate}

\begin{theorem}\label{th:simpleopt}
  Let $[\boldsymbol \sigma _N, \mathcal T_N]=\hbox{\bf \small
    AMFEM}(\mathcal T_H, f_H, \varepsilon, \theta, 0, \mu)$, and
  $\tilde{\boldsymbol \sigma} = \mcal L^{-1}f_H$. If
  $\tilde{\boldsymbol \sigma} \in \mathcal A^s$ and $0<\theta<
  C_2/C_1$, then for any $\varepsilon>0$, {\bf \small AMFEM} will
  terminated in finite steps and
 \begin{equation}\label{solutionopt}
   \|\tilde{\boldsymbol \sigma} -\boldsymbol \sigma _N\|\leq \varepsilon, \quad  \hbox{and}\quad \# \mathcal T_N-\#\mathcal T_0\leq C\|\boldsymbol \sigma\|_{\mathcal A^s}^{1/s}\varepsilon ^{-1/s}.
\end{equation}
\end{theorem}
\begin{proof}
  It is identical to the proof of Theorem 5.3 in
  \cite{Stevenson.R2007} using three ingredients listed above.
\end{proof}

\begin{theorem}
  For any $f\in L^2(\Omega)$, a shape regular triangulation $\mathcal
  T_0$ and $\varepsilon >0$. Let $\boldsymbol \sigma=\mcal L^{-1}f$
  and $[\boldsymbol \sigma _N, \mathcal T_N]= \hbox{\bf \small AMFEM
  }(\mathcal T_H, f_H, \varepsilon/2, 0, 1)$ where $[\mathcal T_H,
  f_H]=\hbox{\bf \small APPROX }(f, \mathcal T_0,\varepsilon/2)$. If
  $\boldsymbol \sigma \in \mathcal A^s$ and $f\in \mathcal A^{s}_o$,
  then
$$
\|\boldsymbol \sigma - \boldsymbol \sigma _N\|\leq C\big
(\|\boldsymbol \sigma \|_{\mathcal A^s}+\|f\|_{\mathcal A_o^s}\big
)(\#\mathcal T_N-\#\mathcal T_0)^{-s}.
$$ 
\end{theorem}
\begin{proof}
  Let $\tilde {\boldsymbol \sigma} = \mcal L^{-1}f_H$. By Theorem
  \ref{th:stabilityosc} and \ref{th:BDD}, we have
\begin{equation}\label{eq:11}
  \|\boldsymbol \sigma -\tilde {\boldsymbol \sigma}\|\leq \varepsilon /2,\quad \hbox{and}\quad \#\mathcal T_H-\#\mathcal T_0\leq C \|f\|_{\mathcal A^{s}_o}^{1/s}\varepsilon ^{-1/s}.
\end{equation}
It is easy to show, by the definition of $\mathcal A^s$, if
$\boldsymbol \sigma \in \mathcal A^s$, then $\tilde {\boldsymbol
  \sigma} \in \mathcal A^s$ and
$$
\|\tilde {\boldsymbol \sigma} \|_{\mathcal A^s}\leq \|\boldsymbol
\sigma \|_{\mathcal A^s}+\|f\|_{\mathcal A^s_o}
$$
We then apply Theorem \ref{th:simpleopt} to $\tilde {\boldsymbol
  \sigma}$ to obtain
\begin{equation}\label{eq:12}
  \|\tilde {\boldsymbol \sigma} - \boldsymbol \sigma _N\|\leq \varepsilon/2 \quad \hbox{ and }\quad
  \# \mathcal T_N-\#\mathcal T_H\leq C\|\tilde {\boldsymbol \sigma}\|_{\mathcal A^s}^{1/s}\varepsilon ^{-1/s}.
\end{equation}
Combining (\ref{eq:11}) and (\ref{eq:12}) we get
$$
\|\boldsymbol \sigma -\boldsymbol \sigma _N\|\leq \|\boldsymbol
\sigma-\tilde {\boldsymbol \sigma}\|+\|\tilde {\boldsymbol
  \sigma}-\boldsymbol \sigma_{N}\|\leq \varepsilon
$$
and
$$
\varepsilon \leq C (\# \mathcal T_N-\#\mathcal T_0)^{-s}\big
(\|\boldsymbol \sigma\|_{\mathcal A^s}+\|f\|_{\mathcal A^{s}_o}\big ).
$$
The desired result then follows.
\end{proof}

\section{Conclusion and future work}
In this paper, we have designed and analyzed convergent adaptive mixed finite element methods with optimal complexity for arbitrary order Raviart-Thomas and Brezzi-Douglas-Marini elements. Although the results are presented in two dimensions, most of them are dimensional independent. For example, the discrete stability result, Theorem \ref{th:stabilityosc}, holds in arbitrary dimensions without any modification of the proof. 

The proof for the upper bound of the error estimator (Theorem \ref{th:Alonso} and \ref{th:upperbound}), however, cannot be generalized to three dimensions in a straightforward way. In the proof, we use a special fact that in two dimensions, $H(\curl)$ is as smooth as $H^1$ since in two dimensions curl operator is just a rotation of gradient operator. To overcome this difficulty, we need to use a regular decomposition instead of Helmholtz decomposition. Note that discrete regular decomposition for corresponding finite element spaces is developed recently by Hiptmair and Xu \cite{Hiptmair.R;Xu.J2006}. We could use these techniques to prove the convergence and optimality of adaptive mixed finite element methods in three and higher dimensions.

\subsection*{Acknowledgement} The authors would like to thank
Dr. Guzman for the discussion on the simplification of the proof of the
discrete stability result and Prof. Nochetto for the simplification of
convergence analysis without using discrete lower bound.

\bibliographystyle{abbrv}
\bibliography{../bib/books,../bib/papers,../bib/mjh,../bib/library,../bib/ref-gn,../bib/coupling,../bib/pnp}

\bibliographystyle{abbrv}
\begin{thebibliography}{10}

\bibitem{Alonso.A1996}
A.~Alonso.
\newblock Error estimators for a mixed method.
\newblock {\em Numerische Mathematik}, 74(4):385--395, 1996.

\bibitem{Arnold.D2004}
D.~N. Arnold.
\newblock Differential complexes and numerical stability.
\newblock {\em Plenary address delivered at ICM 2002 International Congress of
  Mathematicians}, 2004.

\bibitem{Arnold.D;Brezzi.F1985}
D.~N. Arnold and F.~Brezzi.
\newblock Mixed and nonconforming finite element methods: Implementation,
  postporcessing and error estimates.
\newblock {\em RAIRO Model Math. Anal. Numer.}, 19:7--32, 1985.

\bibitem{Arnold.D;Falk.R;Winther.R2000}
D.~N. Arnold, R.~S. Falk, and R.~Winther.
\newblock Multigrid in ${H}(div)$ and ${H}(curl)$.
\newblock {\em Numerische Mathematik}, 85:197--218, 2000.

\bibitem{Arnold.D;Falk.R;Winther.R2005}
D.~N. Arnold, R.~S. Falk, and R.~Winther.
\newblock Differential complexes and stability of finite element methods. {I}.
  the de {R}ham complex.
\newblock {\em preprint}, 2005.

\bibitem{Arnold.D;Falk.R;Winther.R2006}
D.~N. Arnold, R.~S. Falk, and R.~Winther.
\newblock Finite element exterior calculus, homological techniques, and
  applications.
\newblock {\em Acta Numerica}, pages 1--155, 2006.

\bibitem{Arnold.D;Scott.L;Vogelius.M1988}
D.~N. Arnold, L.~R. Scott, and M.~Vogelius.
\newblock Regular inversion of the divergence operator with {D}irichlet
  boundary conditions on a polygon.
\newblock {\em Ann. Scuola Norm. Sup. Pisa Cl. Sci. (4)}, 15(2):169--192
  (1989), 1988.

\bibitem{Babuska.I;Vogelius.M1984}
I.~Babu{\v s}ka and M.~Vogelius.
\newblock Feeback and adaptive finite element solution of one-dimensional
  boundary value problems.
\newblock {\em Numerische Mathematik}, 44:75--102, 1984.

\bibitem{Bacuta.C;Bramble.J;Xu.J2002}
C.~Bacuta, J.~H. Bramble, and J.~Xu.
\newblock Regularity estimates for elliptic boundary value problems in besov
  spaes.
\newblock {\em Mathematics of Computation}, 72(244):1577--1595, 2002.

\bibitem{Bacuta.C;Bramble.J;Xu.J2003}
C.~Bacuta, J.~H. Bramble, and J.~Xu.
\newblock Regularity estimates for elliptic boundary value problems with smooth
  data on polygonal domains.
\newblock {\em Numerische Mathematik}, 11(2):75--94, 2003.

\bibitem{Bank.R;Sherman.A;Weiser.A1983}
R.~E. Bank, A.~H. Sherman, and A.~Weiser.
\newblock Refinement algorithms and data structures for regular local mesh
  refinement.
\newblock In {\em Scientific Computing}, pages 3--17. IMACS/North-Holland
  Publishing Company, Amsterdam, 1983.

\bibitem{Binev.P;Dahmen.W;DeVore.R2004}
P.~Binev, W.~Dahmen, and R.~DeVore.
\newblock Adaptive finite element methods with convergence rates.
\newblock {\em Numerische Mathematik}, 97(2):219--268, 2004.

\bibitem{Binev.P;Dahmen.W;DeVore.R;Petrushev.P2002}
P.~Binev, W.~Dahmen, R.~DeVore, and P.~Petrushev.
\newblock Approximation classes for adaptive methods.
\newblock {\em Serdica Math. J}, 28:391--416, 2002.

\bibitem{Binev.P;DeVore.R2004}
P.~Binev and R.~DeVore.
\newblock Fast computation in adaptive tree approximation.
\newblock {\em Numerische Mathematik}, 97:193--217, 2004.

\bibitem{Bochev.P;Gunzburger.M2005}
P.~Bochev and M.~Gunzburger.
\newblock On least-sequares finite element methods for the poisson equation and
  their connection to the {Dirichlet} and {Kelvin} principles.
\newblock {\em SIAM Journal on Numerical Analysis}, 43(1):340--362, 2005.

\bibitem{Bossavit.A1988}
A.~Bossavit.
\newblock Whitney forms: a class of finite elements for three-dimensional
  computations in electromagnetism.
\newblock {\em Science, Measurement and Technology, IEE Proceedings},
  135(8):493--500, Nov 1988.

\bibitem{Braess.D;Verfurth.R1990}
D.~Braess and R.~Verf{\"u}rth.
\newblock Multigrid methods for nonconforming finite element methods.
\newblock {\em SIAM Journal on Numerical Analysis}, 27:979--986, 1990.

\bibitem{Bramble.J;Pasciak.J;Xu.J1988}
J.~H. Bramble, J.~E. Pasciak, and J.~Xu.
\newblock The analysis of multigrid algorithms for nonsymmetric and indefinite
  elliptic problems.
\newblock {\em Mathematics of Computation}, 51:389--414, 1988.

\bibitem{Brandts.J1994a}
J.~H. Brandts.
\newblock {\em Superconvergence phenomena in finite element methods}.
\newblock PhD thesis, Utrecht University, 1994.

\bibitem{Brenner.S1992}
S.~C. Brenner.
\newblock A multigrid algorithm for the lowest-order {R}aviart-{T}homas mixed
  triangular finite element method.
\newblock {\em SIAM Journal on Numerical Analysis}, 29:647--678, 1992.

\bibitem{Brenner.S1996}
S.~C. Brenner.
\newblock Two-level additive {S}chwarz preconditioners for nonconforming finite
  element methods.
\newblock {\em Mathematics of Computation}, 65:897--921, 1996.

\bibitem{Brenner.S;Scott.L2002}
S.~C. Brenner and L.~R. Scott.
\newblock {\em The mathematical theory of finite element methods}, volume~15 of
  {\em Texts in Applied Mathematics}.
\newblock Springer-Verlag, New York, second edition, 2002.

\bibitem{Brezzi.F;Douglas.J;Marini.L1985}
F.~Brezzi, J.~Douglas, and L.~D. Marini.
\newblock Two families of mixed finite elements for second order elliptic
  problems.
\newblock {\em Numerische Mathematik}, 47(2):217--235, 1985.

\bibitem{Brezzi.F;Fortin.M1991}
F.~Brezzi and M.~Fortin.
\newblock {\em Mixed and hybrid finite element methods}.
\newblock Springer-Verlag, 1991.

\bibitem{Carstensen.C1997}
C.~Carstensen.
\newblock A posteriori error estimate for the mixed finite element method.
\newblock {\em Mathematics of Computation}, 66:465--476, 1997.

\bibitem{Carstensen.C;Hoppe.R2006}
C.~Carstensen and R.~H.~W. Hoppe.
\newblock Error reduction and convergence for an adaptive mixed finite element
  method.
\newblock {\em Math. Comp.}, 75(255):1033--1042 (electronic), 2006.

\bibitem{Cascon.J;Kreuzer.C;Nochetto.R;Siebert.K2007}
J.~M. Casc\'on, C.~Kreuzer, R.~H. Nochetto, and K.~G. Siebert.
\newblock Quasi-optimal convergence rate for an adaptive finite element method.
\newblock {\em Preprint 9, University of Augsburg}, 2007.

\bibitem{Chen.L2006a}
L.~Chen.
\newblock Short implementation of bisection in {MATLAB}.
\newblock {\em report}, 2006.

\bibitem{Chen.L;Xu.J2007}
L.~Chen and J.~Xu.
\newblock Topics on adaptive finite element methods.
\newblock In T.~Tang and J.~Xu, editors, {\em Adaptive Computations: Theory and
  Algorithms}, pages 1--31. Science Press, Beijing, 2007.

\bibitem{Ciarlet.P1978}
P.~G. Ciarlet.
\newblock {\em The Finite Element Method for Elliptic Problems}, volume~4 of
  {\em Studies in Mathematics and its Applications}.
\newblock North-Holland Publishing Co., Amsterdam-New York-Oxford, 1978.

\bibitem{Cockburn.B;Gopalakrishnan.J2004}
B.~Cockburn and J.~Gopalakrishnan.
\newblock A characterization of hybridized mixed methods for second order
  elliptic problems.
\newblock {\em SIAM Journal on Numerical Analysis}, 42(1):283--301, 2004.

\bibitem{Cockburn.B;Gopalakrishnan.J2005a}
B.~Cockburn and J.~Gopalakrishnan.
\newblock Error analysis of variable degree mixed methods for elliptic problems
  via hybridization.
\newblock {\em Mathematics of Computation}, 74(252):1653--1677, 2005.

\bibitem{Cockburn.B;Gopalakrishnan.J2005}
B.~Cockburn and J.~Gopalakrishnan.
\newblock New hybridization techniques.
\newblock {\em GAMM-Mitt.}, 28(2):154--182, 2005.

\bibitem{Dahlke.S1999}
S.~Dahlke.
\newblock Besov regularity for elliptic boundary value problems on polygonal
  domains.
\newblock {\em Appl. Math. Lett.}, 12:31--36, 1999.

\bibitem{Dahlke.S;DeVore.R1997}
S.~Dahlke and R.~A. DeVore.
\newblock Besov regularity for elliptic boundary value problems.
\newblock {\em Comm. Partial Differential Equations}, 22(1\&2):1--16, 1997.

\bibitem{Dorfler.W1996}
W.~D\"orfler.
\newblock A convergent adaptive algorithm for {Poisson}'s equation.
\newblock {\em SIAM Journal on Numerical Analysis}, 33:1106--1124, 1996.

\bibitem{Duran.R;Muschietti.M2001}
R.~G. Dur{\'a}n and M.~A. Muschietti.
\newblock An explicit right inverse of the divergence operator which is
  continuous in weighted norms.
\newblock {\em Studia Math.}, 148(3):207--219, 2001.

\bibitem{Fix.G;Gunzburger.M;Nicolaides.R1981}
G.~J. Fix, M.~D. Gunzburger, and R.~A. Nicolaides.
\newblock On mixed finite element methods for first order elliptic systems.
\newblock {\em Numerische Mathematik}, 37(1):29--48, 1981.

\bibitem{Gatica.G;Maischak.M2004}
G.~N. Gatica and M.~Maischak.
\newblock A posteriori error estimates for the mixed finite element method with
  lagrange multipliers.
\newblock {\em Numer. Methods Partial Differential Equations}, 21(3):421 --
  450, 2004.

\bibitem{Gopalakrishnan.J2003}
J.~Gopalakrishnan.
\newblock A {Schwarz} preconditioner for a hybridized mixed method.
\newblock {\em Computational Methods In Applied Mathematics}, 3(1):116---134,
  2003.

\bibitem{Hiptmair.R1999a}
R.~Hiptmair.
\newblock Canonical construction of finite elements.
\newblock {\em Mathematics of Computation}, 68:1325--1346, 1999.

\bibitem{Hiptmair.R;Xu.J2006}
R.~Hiptmair and J.~Xu.
\newblock Nodal auxiliary space preconditioning in {H(curl)} and {H(div)}
  spaces.
\newblock Research report no. 2006-09, ETH, Zurich, Switzerland, 2006.

\bibitem{Hoppe.R;Wohlmuth.B1997}
R.~H.~W. Hoppe and B.~Wohlmuth.
\newblock Adaptive multilevel techniques for mixed finite element
  discretizations of elliptic boundary value problems.
\newblock {\em SIAM Journal on Numerical Analysis}, 34(4):1658--1681, aug 1997.

\bibitem{Larson.M;Maqvist.A2005}
M.~G. Larson and A.~Maqvist.
\newblock A posteriori error estimates for mixed finite element approximations
  of elliptic problems.
\newblock {\em Preprint}, 2005.

\bibitem{Lovadina.C;Stenberg.R2006}
C.~Lovadina and R.~Stenberg.
\newblock Energy norm a posteriori error estimates for mixed finite element
  methods.
\newblock {\em Mathemathics of Computation}, Primary 65N30, 2006.

\bibitem{Mitchell.W1989}
W.~F. Mitchell.
\newblock A comparison of adaptive refinement techniques for elliptic problems.
\newblock {\em ACM Transactions on Mathematical Software (TOMS) archive},
  15(4):326 -- 347, 1989.

\bibitem{Mitchell.W1992}
W.~F. Mitchell.
\newblock Optimal multilevel iterative methods for adaptive grids.
\newblock {\em SIAM Journal on Scientific and Statistical Computing},
  13:146--167, 1992.

\bibitem{Morin.P;Nochetto.R;Siebert.K2000}
P.~Morin, R.~Nochetto, and K.~Siebert.
\newblock Data oscillation and convergence of adaptive {FEM}.
\newblock {\em SIAM Journal on Numerical Analysis}, 38(2):466--488, 2000.

\bibitem{Morin.P;Nochetto.R;Siebert.K2002}
P.~Morin, R.~H. Nochetto, and K.~G. Siebert.
\newblock Convergence of adaptive finite element methods.
\newblock {\em SIAM Review}, 44(4):631--658, 2002.

\bibitem{Morin.P;Nochetto.R;Siebert.K2003}
P.~Morin, R.~H. Nochetto, and K.~G. Siebert.
\newblock Local problems on stars: A posteriori error estimators, convergence,
  and performance.
\newblock {\em Mathematics of Computation}, 72:1067--1097, 2003.

\bibitem{Morin.P;Siebert.K;Veeser.A2007a}
P.~Morin, K.~G. Siebert, and A.~Veeser.
\newblock A basic convergence result for conforming adaptive finite elements.
\newblock {\em Preprint, University of Augsburg}, 2007.

\bibitem{Oswald.P1997}
P.~Oswald.
\newblock Intergrid transfer operators and multilevel preconditioners for
  nonconforming discretizations.
\newblock {\em Applied Numerical Mathematics}, 23(1):139--158, 1997.

\bibitem{Raviart.P;Thomas.J1977}
P.~A. Raviart and J.~Thomas.
\newblock A mixed finite element method fo 2-nd order elliptic problems.
\newblock In I.~Galligani and E.~Magenes, editors, {\em Mathematical aspects of
  the Finite Elements Method}, Lectures Notes in Math. 606, pages 292--315.
  Springer, Berlin, 1977.

\bibitem{Rivara.M1984a}
M.~C. Rivara.
\newblock Design and data structure for fully adaptive, multigrid finite
  element software.
\newblock {\em ACM Trans. Math. Soft.}, 10:242--264, 1984.

\bibitem{Rivara.M1984}
M.~C. Rivara.
\newblock Mesh refinement processes based on the generalized bisection of
  simplices.
\newblock {\em SIAM Journal on Numerical Analysis}, 21:604--613, 1984.

\bibitem{Rusten.T;Vassilevski.P;Winther.R1996}
T.~Rusten, P.~S. Vassilevski, and R.~Winther.
\newblock Interior penalty preconditioners for mixed finite element
  approximations of elliptic problems.
\newblock {\em Mathemathics of Computation}, 65:447--466, 1996.

\bibitem{Scott.R;Zhang.S1990}
R.~Scott and S.~Zhang.
\newblock Finite element interpolation of nonsmooth functions satisfying
  boundary conditions.
\newblock {\em Mathematics of Computation}, 54:483--493, 1990.

\bibitem{Sewell.E1972}
E.~G. Sewell.
\newblock Automatic generation of triangulations for piecewise polynomial
  approximation.
\newblock In {\em Ph. D. dissertation}. Purdue Univ., West Lafayette, Ind.,
  1972.

\bibitem{Stevenson.R2007}
R.~Stevenson.
\newblock Optimality of a standard adaptive finite element method.
\newblock {\em Found. Comput. Math.}, 7(2):245--269, 2007.

\bibitem{Stevenson.R2008}
R.~Stevenson.
\newblock The completion of locally refined simplicial partitions created by
  bisection.
\newblock {\em Mathemathics of Computation}, 77:227--241, 2008.

\bibitem{Verfurth.R1996}
R.~Verf{\"u}rth.
\newblock {\em A review of a posteriori error estimation and adaptive mesh
  refinement tecniques}.
\newblock B. G. Teubner, 1996.

\bibitem{Wohlmuth.B;Hoppe.R1999}
B.~I. Wohlmuth and R.~H.~W. Hoppe.
\newblock A comparison of a posteriori error estimators for mixed finite
  element discretizations by raviart-thomas elements.
\newblock {\em Mathematics of Computation}, 82:253--279, 1999.

\end{thebibliography}

\clearpage

\vspace*{0.5cm}

\end{document}